\documentclass[12pt,a4paper,english]{amsart}
\usepackage{amsfonts}
\usepackage{amsthm}
\usepackage{amsmath}
\usepackage{amscd}
\usepackage{array}
\usepackage{amssymb}
\usepackage[latin2]{inputenc}
\usepackage{t1enc}
\usepackage[mathscr]{eucal}
\usepackage{indentfirst}
\usepackage{hyperref}
\usepackage{graphicx}
\usepackage{subfig}
\usepackage{caption}
\usepackage{etoolbox}
\usepackage{graphics}
\usepackage{pict2e}
\usepackage{epic}
\numberwithin{equation}{section}
\usepackage[margin=2.4cm]{geometry}
\usepackage{epstopdf} 
\usepackage{tikz} 
\usepackage{float}

\usetikzlibrary{automata}

\theoremstyle{plain}
\newtheorem{teo}{Theorem}[section]

\newtheorem{coro}[teo]{Corollary}
\newtheorem{prop}[teo]{Proposition}

 \theoremstyle{definition}
\newtheorem{defi}[teo]{Definition}
\newtheorem{conj}[teo]{Conjecture}
\newtheorem{obs}[teo]{Remark}
\newtheorem{?}[teo]{Problem}
\newtheorem{ej}[teo]{Example}

\newcommand{\rk}{\operatorname{rk}}
\newcommand{\R}{\mathbb{R}}
\newcommand{\Z}{\mathbb{Z}}

\allowdisplaybreaks

\begin{document}

\title[Integer points on independence polytopes and half-open hypersimplices]{Integer point enumeration on independence polytopes and half-open hypersimplices}

\author[L. Ferroni]{Luis Ferroni}
\thanks{The author is supported by the Marie Sk{\l}odowska-Curie PhD fellowship as part of the program INdAM-DP-COFUND-2015.}

\address{Universit\`a di Bologna, Dipartimento di Matematica, Piazza di Porta San Donato, 5, 40126 Bologna BO - Italia} 

\email{ferroniluis@gmail.com}


\begin{abstract} 
    In this paper we investigate the Ehrhart Theory of the independence matroid polytope of uniform matroids. It is proved that these polytopes have an Ehrhart polynomial with positive coefficients. To do that, we prove that indeed all half-open-hypersimplices are Ehrhart positive, and tile disjointly our polytope using them.
\end{abstract}

\maketitle

\section{Introduction}

Let us consider a polytope $\mathscr{P}\subseteq \R^n$ having vertices with integer coordinates. The function counting the number of integer points on each integral dilation of $\mathscr{P}$,
    \[ i(\mathscr{P}, t) := \# (t\mathscr{P} \cap \Z^n)\]
happens to be a polynomial which in the literature is called the \textit{Ehrhart polynomial} of $\mathscr{P}$ \cite{ehrhart,beck}.

The Ehrhart polynomial of the members of some basic families of polytopes, such as \textit{regular simplices}, \textit{cross-polytopes} and \textit{hypercubes} can be calculated by hand \cite{beck}. More complicated formulas exist for other families such as \textit{Pitman-Stanley polytopes} \cite{stanleypitman}, \textit{$\mathcal{Y}$-generalized permutohedra} \cite{postnikov} and some subfamilies of \textit{flow polytopes} \cite{meszaros}.\\

A natural question that arises when studying the Ehrhart polynomial of a polytope is whether its coefficients are positive. In \cite{liu} Fu Liu gave an extensive list of the polytopes that are known to have this property. In \cite{fuliu} Castillo and Liu conjectured:

\begin{conj}\label{genperpos}
    If $\mathscr{P}$ is an integral generalized permutohedron then its Ehrhart polynomial has positive coefficients.
\end{conj}

A generalized permutohedron is a polytope that has all of its edges (i.e. its one dimensional faces) parallel to some vector of the form $e_i - e_j$, where $e_i$ is the canonical vector in $\R^n$ having a one on the $i$-th position and zeros elsewhere, and analogously for $e_j$.

This assertion implies that, for instance, all matroid polytopes are Ehrhart positive, as was also conjectured by De Loera et al \cite{deloera}. In fact, matroid polytopes are exactly those generalized permutohedra having vertices with $0/1$-vertices as was follows from the characterization of matroid polytopes given in \cite{GGMS}.\\

Also, in \cite{ardila} generalized permutohedra are characterized as the polytopes arising as a Minkowski signed sum of dilated simplices that satisfies a certain modularity property. The subfamily of $\mathcal{Y}$-generalized permutohedra can then be recovered as that of strictly positive Minkowski sum of dilated simplices (always satisfying the aforementioned modularity property).

This large subclass of generalized permutohedra consist of Ehrhart positive polytopes, as was proved by Postnikov \cite{postnikov}. However, there are examples of matroid polytopes that are not $\mathcal{Y}$-generalized permutohedra \cite{ardila}. In fact, from the work of Ardila, Benedetti and Doker it follows that if a connected matroid has a rank and corank greater than 1, it is not a $\mathcal{Y}$-generalized permutohedron (the signed $\beta$ invariant defined therein has to be negative for some contraction of $M$).\\

Recall that the $(k,n)$-hypersimplex is defined as:
    \[ \Delta_{k,n} = \left\{x\in [0,1]^n : \sum_{i=1}^n x_i = k\right\}.\]
This polytope appears in several diverse contexts within algebraic combinatorics. For what occupies us here, this is the basis polytope of the uniform matroid $U_{k,n}$, and as such, it is a generalized permutohedron. As we stated before, it is not a $\mathcal{Y}$-generalized permutohedron when $1 < k < n - 1$. In spite of that, recently the author proved that it has an Ehrhart polynomial with positive coefficients \cite{ferroni}.\\

Very little is known about the Ehrhart polynomials of the independence polytope of matroids. We will prove the following easy but sometimes disregarded fact:

\begin{teo}
    The independence matroid polytope $\mathscr{P}_I$ of a matroid is integrally equivalent to a generalized permutohedron that we call $\widetilde{\mathscr{P}_I}$.
\end{teo}

Hence, Conjecture \ref{genperpos} would imply that independence matroid polytopes are Ehrhart positive. This polytope $\widetilde{P}_I$ is not a $\mathcal{Y}$-generalized permutohedron, because of the same reason that the basis polytope of a matroid is not.\\

All that was said so far was to make the following result plausible.

\begin{teo}
    The independence matroid polytope of the uniform matroid $U_{k,n}$, which is given by
        \[\mathscr{P}_I (U_{k,n}) = \left\{ x\in [0,1]^n : \sum_{i=1}^n x_i \leq k\right\},\]
    is Ehrhart positive. 
\end{teo}

We are going to prove it. The key step is to prove that all \textit{half-open hypersimplices} introduced by Nan Li \cite{nanli}, defined by
    \[ \Delta'_{k,n} = \left\{x\in [0,1]^{n-1} : k-1 < \sum_{i=1}^{n-1} x_i \leq k \right\}\]
for $k>1$ and $\Delta'_{1,n}:=\Delta_{1,n}$ are Ehrhart positive. We use them to tile disjointly the independence polytopes of uniform matroids. This technique introduces a new tool to prove the Ehrhart positivity of a polytope.

\section{Independence Matroid Polytopes}

\begin{defi}
    Let $M=(E=\{1,\ldots,n\},\rk)$ be a matroid of rank $k$ and cardinality $n$. For each subset $A$ of $E$, let us denote:
            \[ e_A := \sum_{i\in A} e_i,\]
    where $e_i$ is the $i$-th canonical vector in $\mathbb{R}^n$.
    The \textit{basis polytope} $\mathscr{P}(M)$ and the \textit{independence polytope} $\mathscr{P}_I(M)$ are defined, respectively, as:
        \begin{align*}
        \mathscr{P}(M) &:= \operatorname{convex\, hull}\{e_B : B\subseteq E \text{ is a basis}\},\\
        \mathscr{P}_I(M) &:= \operatorname{convex\, hull}\{e_I : I\subseteq E \text{ is independent}\}.
        \end{align*}
\end{defi}

Evidently the basis polytope $\mathscr{P}(M)$ is a facet of $\mathscr{P}_I(M)$. Also, notice that any of these two polytopes does in fact determine the matroid $M$. There exist characterizations for all polytopes arising as a basis polytope \cite{GGMS} or as an independence polytope of a matroid \cite{schrijver}.

\begin{teo}
    A polytope $\mathscr{P}\subseteq \R^n$ is the basis polytope of a matroid with $n$ elements if and only if $\mathscr{P}$ satisfies the following two conditions:
    \begin{itemize}
        \item All the vertices of $\mathscr{P}$ have $0/1$ coordinates.
        \item All the edges of $\mathscr{P}$ are of the form $e_i-e_j$.
    \end{itemize}
\end{teo}

\begin{teo}
        Let $\mathscr{P}\subseteq \R^n$ be the independence polytope of a matroid with $n$ elements. Then:
    \begin{itemize}
        \item All the vertices of $\mathscr{P}$ have $0/1$ coordinates.
        \item All the edges of $\mathscr{P}$ are of the form $e_i-e_j$, $e_i$ or $-e_i$.
    \end{itemize}
\end{teo}

A \textit{generalized permutohedron} is a polytope $\mathscr{P}\subseteq \R^n$ that has all of its edges parallel to $e_i-e_j$. It is evident that basis polytopes of matroids are a subfamily of generalized permutohedra (precisely those that have vertices with 0/1 coordinates).

\begin{defi}
    Let $M$ be a matroid of rank $k$ and cardinality $n$. We defined the \textit{lifted independence polytope} of $M$ as the polytope $\widetilde{\mathscr{P}}_I(M)\subseteq \R^{n+1}$ given by:
        \[ \widetilde{\mathscr{P}}_I(M) := \text{convex hull} \{ (e_I, k - \rk(I)) : I\subseteq E \text{ is independent}\}.\]
\end{defi}

It is evident that $\mathscr{P}_I$ and $\widetilde{\mathscr{P}}_I(M)$ are integrally equivalent, in fact the map $\mathscr{P}_I(M) \to \widetilde{P}_I(M)$ is a unimodular equivalence. In particular they have the same relative volume and the same Ehrhart polynomial. Moreover, it is straightforward to prove the following result:

\begin{teo}
    For every matroid $M$, the lifted independence polytope $\widetilde{\mathscr{P}}_I(M)$ is a generalized permutohedron. 
\end{teo}

\begin{proof}
    Let us pick two adjacent vertices $v$ and $w$ in $\mathscr{\widetilde{P}}_I(M)$. They are of the form:
        \begin{align*}
            v &= (e_{I_1}, k - \rk(I_1),\\
            w &= (e_{I_2}, k - \rk(I_2)),
        \end{align*}
    for some independent sets $I_1$ and $I_2$ of $M$. Moreover, the vertices $e_{I_1}$ and $e_{I_2}$ are adjacent in $\mathscr{P}_I(M)$. There are two cases:
    \begin{itemize}
        \item If $|I_1| = |I_2|$, then $\rk(I_1) = |I_1| = |I_2| = \rk(I_2)$. Since $e_{I_1}$ and $e_{I_2}$ are adjacent in $\mathscr{P}_I(M)$ the only possibility is that $e_{I_1} - e_{I_2} = e_i - e_j$ for some $i,j$. In particular, $v - w = (e_i - e_j, 0)$. 
        \item If $|I_1| \neq |I_2|$, assume without loss of generality that $|I_1| < |I_2|$. The condition of $e_{I_1}$ and $e_{I_2}$ being adjacent in $\mathscr{P}_I(M)$ implies that $e_{I_2} - e_{I_1} = e_i$ for some $i$. This says that $I_2 = I_1 \sqcup \{i\}$. In particular $\rk(I_2) = \rk(I_1)+1$, and then:
            \[ v - w = (e_i, -1),\]
        which has the desired form.\qedhere
    \end{itemize}
\end{proof}

\begin{obs}
    This result was stated implicitly in \cite{ardila}. We include it explicitly here to simplify future referencing and motivate the main results. We want to emphasize that $\mathscr{P}(M)$ and $\widetilde{\mathscr{P}}_I(M)$ are \textit{not} $\mathcal{Y}$-generalized permutohedra when $M$ is connected and of rank and corank greater than 1.
\end{obs}

\section{The main results}

In \cite{nanli}, Nan Li introduced half open hypersimplices $\Delta'_{k,n}$:
  \begin{equation} \label{defhalfopen}
  \Delta'_{k,n} := \left\{x\in [0,1]^{n-1} : k-1 < \sum_{i=1}^{n-1} x_i \leq k \right\}.\end{equation}
for $k>1$ and $\Delta'_{1,n}:=\Delta_{1,n}$. This was done in the context of studying the Ehrhart $h^*$-polynomial of the hypersimplex $\Delta_{k,n}$. 

The Ehrhart polynomial of $\Delta'_{k,n}$ can be calculated in terms of Ehrhart polynomial of two hypersimplices.

\begin{prop}\label{ehrharthalfopen}
    If $1 < k < n - 1$, then:
        \[ i(\Delta'_{k,n}, t) = i(\Delta_{k,n}, t) - i(\Delta_{k-1,n-1}, t).\]
\end{prop}

\begin{proof}
    Observe $\Delta_{k,n}$ can be seen as the set of points in $[0,1]^{n-1}$ that have sum of coordinates in the interval $[k-1,k]$. If we exclude the possibility of the sum of coordinates being equal to $k-1$, then we are essentially erasing the hypersimplex $\Delta_{k-1,n-1}$.
\end{proof}

A useful fact is that the description given in equation \eqref{defhalfopen} of these objects shows that they are polytopes with some missing faces. Thus we can use them to tile a polytope in a disjoint fashion.\\

Explicitly, for the uniform matroid $U_{k,n}$, the independence polytope is given by:
    \[ \mathscr{P}_I(U_{k,n}) = \left\{ x \in [0,1]^n : \sum_{i=1}^n x_i \leq k \right\}.\]

It is evident that:
    \begin{equation}\label{decompo}
    \mathscr{P}_I(U_{k,n}) = \Delta_{1,n+1}' \sqcup \Delta'_{2,n+1} \sqcup \cdots \sqcup \Delta'_{k,n+1},\end{equation}
where the symbol $\sqcup$ stands for disjoint union. Hence, if we prove that each of these half open hypersimplices is Ehrhart positive, we can conclude so for the independence matroid polytope of the uniform matroid $U_{k,n}$.

This approach can be extended to all polytopes that can be tiled using (dilations of) half-open hypersimplices. The author leaves as a question what polytopes can be tiled in this way.

Recall from \cite{ferroni} the definition of weighted Lah number.

\begin{defi}
	Let $\pi$ be a partition of the set $\{1,\ldots,n\}$ into $m$ linearly ordered blocks. We define the \textit{weight of $\pi$} by the following formula:
		\[ w(\pi) := \sum_{b\in\pi} w(b),\]
	where $w(b)$ is the number of elements in $b$ that are smaller (as positive integers) than the first element in $b$.
\end{defi}

\begin{defi}
	We define the \textit{weighted Lah Numbers} $W(\ell,n,m)$ as the number of partitions of weight $\ell$ of $\{1,\ldots,n\}$ into exactly $m$ linearly ordered blocks. We call $\mathscr{W}(\ell,n,m)$ the family of all such partitions.
\end{defi}

\begin{ej}\label{ejemplito}
	These are the partitions of the set $\{1,2,3\}$ into $2$ ordered blocks:
		\[ \{(1,2),(3)\}, \{(2,1),(3)\},\]
		\[ \{(1,3),(2)\}, \{(3,1),(2)\},\]
		\[ \{(2,3),(1)\}, \{(3,2),(1)\}.\]
	For each of them, we have:
		\[ w(\{(1,2),(3)\}) = 0 + 0 = 0,\;\; w(\{(2,1),(3)\}) = 1 + 0 = 1,\]
		\[ w(\{(1,3),(2)\}) = 0 + 0 = 0,\;\;  w(\{(3,1),(2)\}) = 1 + 0 = 1,\]
		\[ w(\{(2,3),(1)\}) = 0 + 0 = 0,\;\;  w(\{(3,2),(1)\}) = 1 + 0 = 1.\]
	Note that there are exactly $3$ of these partitions of weight $0$ and exactly $3$ of weight $1$. This says that $W(0,3,2)=3$ and $W(1,3,2)=3$.
\end{ej}

In \cite[Proposition 3.10]{ferroni} a recurrence for $W(\ell,n,m)$ is stated. Here we will need one that is very similar.

\begin{prop}
    \[W(\ell,n,m) = (n-1) W(\ell,n-1,m) + \sum_{j=0}^{n-1} \binom{n-1}{j} j! W(\ell-j,n-1-j,m-1).\]
\end{prop}

\begin{proof}
    Every $\pi\in\mathscr{W}(\ell,n,m)$ has the number $n$ inside a block. If this number is \textit{not} the first element of its block, this means that if we remove it from $\pi$ we end up getting an element of $\mathscr{W}(\ell,n-1,m)$. Analogously, we can pick an element of $\mathscr{W}(\ell,n-1,m)$ and reconstruct an element of $\mathscr{W}(\ell,n,m)$ by adjoining the element $n$ in such a way that it is not the first element of a block. There are $n-1$ possibilities of where to put the number $n$ to get an element of $\mathscr{W}(\ell,n,m)$. So we get the first summand.
	
	The remaining cases to consider are those on which $n$ is the first element of its block. In this case we choose $j$ elements to be in this block, and in every possible order of these elements, the block will always have weight $j$. So the remaining $n-j-1$ elements will have to be arranged in $m-1$ blocks of total weight $\ell-j$.
\end{proof}

\begin{coro}\label{corito}
    For each $2\leq m \leq n$ and $0\leq \ell\leq n-m$ one has:
    \[W(\ell,n,m) > (n-1) W(\ell,n-1,m).\]
\end{coro}

\begin{proof}
    From the preceding Proposition, it suffices to show that at least one term of the sum:
    \[\sum_{j=0}^{n-1} \binom{n-1}{j} j! W(\ell-j,n-1-j,m-1),\]
    is nonzero. Notice that taking $j=\ell$ in the above sum yields the term:
    \[ \binom{n-1}{\ell} \ell! W(0,n-1-\ell, m-1).\]
    Notice that $W(0,n-1-\ell, m-1) > 0$ under the constraints on $\ell$, $n$, and $m$. In fact it is equal to the unsigned Stirling number of the first kind: ${n-1-\ell} \brack {m-1}$.
\end{proof}

\begin{teo}
    Let us denote $i(\Delta_{k,n}',t)$ the Ehrhart polynomial of $\Delta_{k,n}'$. Then:
        \[ [t^m] i(\Delta'_{k,n},t) > 0 \text{ for all } 1\leq m \leq n-1.\]
    Also, the constant term is $1$ for $k=1$ and $0$ for $k>1$.
\end{teo}

\begin{proof}
    Notice that $\Delta'_{1,n}=\Delta_{1,n}$, so the case $k=1$ is already settled in \cite{ferroni}. 
    
    From now on, consider $1\leq m \leq n-1$. We know from \cite{ferroni} that the coefficient of degree $m$ of $i(\Delta_{k,n},t)$ is given by:
    \begin{equation}\label{coefi}
        e_{k,n,m} := \frac{1}{(n-1)!} \sum_{\ell=0}^{k-1} W(\ell,n,m+1)A(m,k-\ell-1).
    \end{equation}
    Where $A$ stands for the Eulerian numbers. Proposition \ref{ehrharthalfopen} says that we have to prove:
    \[e_{k,n,m} > e_{k-1,n-1,m}.\] 
    Since the hypersimplices $\Delta_{k,n}$ and $\Delta_{n-k,n}$ are one a reflection of the other, we also know that $e_{k,n,m} = e_{n-k,n,m}$. This reasoning shows that $e_{k-1,n-1,m}=e_{n-k,n-1,m}.$
    So, it suffices to show that:
        \[ e_{n-k,n,m} > e_{n-k,n-1,m}.\]
    However, setting for simplicity $k' = n-k$ and using equation \eqref{coefi}, the last inequality is equivalent to:
        \[ \frac{1}{(n-1)!} \sum_{\ell=0}^{k'-1} W(\ell,n,m+1)A(m,k'-\ell-1) > \frac{1}{(n-2)!} \sum_{\ell=0}^{k'-1} W(\ell,n-1,m+1)A(m,k'-\ell-1)\]
    Which in turn is equivalent to prove that:
        \[ \sum_{\ell=0}^{k'-1} \left(\frac{1}{n-1} W(\ell,n,m+1) - W(\ell,n-1,m+1)\right) A(m,k'-\ell-1) > 0\]
    And as we saw in Corollary \ref{corito}, the term in the parentheses is positive, as desired.
\end{proof}

\begin{teo}
    The independence matroid polytope of the uniform matroid $U_{k,n}$, given by:
        \[\mathscr{P}_I (U_{k,n}) = \left\{ x\in [0,1]^n : \sum_{i=1}^n x_i \leq k\right\}.\]
    is Ehrhart positive. 
\end{teo}

\begin{proof}
    From the disjoint decomposition of equation \eqref{decompo} it follows that:
        \[i\left(\mathscr{P}_I(U_{k,n}\right), t) = \sum_{j=1}^k i(\Delta'_{j,n}, t),\]
    and hence, the independent term is 1, and the rest of them are positive because in each summand on the right one has such positivity.
\end{proof}

\bibliographystyle{plain}
\bibliography{bibliopaper}

\end{document}